\documentclass[a4paper,notitlepage]{article}

\usepackage[latin1]{inputenc} 
\usepackage[T1]{fontenc}
\usepackage{amsmath}
\usepackage{amsthm}
\usepackage{amssymb}   
\usepackage{mathrsfs}  
\usepackage{mathabx}
\usepackage{latexsym}
\usepackage[dvips]{graphicx}

\theoremstyle{plain}
\newtheorem{thm}{Theorem}[section]
\newtheorem{lemma}[thm]{Lemma}

\newtheorem{cor}[thm]{Corollary}

\theoremstyle{definition}
\newtheorem{define}{Definition}[section]
\newtheorem*{define*}{Definition}

\theoremstyle{remark}

\newtheorem*{remark*}{Remark}

\newcommand{\inte}{\ensuremath{\mathbb{Z}}}
\newcommand{\complex}{\ensuremath{\mathbb{C}}}

\newcommand{\C}{\ensuremath{\mathbb{C}}}

\title{
Large supremum norms and small Shannon entropy 
for Hecke eigenfunctions of quantized cat maps}
\author{Rikard Olofsson
\thanks{The author is supported by Knut och Alice Wallenbergs Stiftelse }}

\begin{document}

\maketitle

\begin{abstract}
This paper concerns the behavior of eigenfunctions of quantized cat maps and in 
particular their supremum norm. We observe that for composite integer values of 
$N,$ the inverse of Planck's constant, 
some of the desymmetrized eigenfunctions have very small support and hence 
very large supremum norm. We also prove an entropy estimate and show that our 
functions satisfy equality in this estimate.
In the case when $N$ is a prime power with even 
exponent we 
calculate the supremum norm for a large proportion of all 
desymmetrized eigenfunctions and we find that for a given $N$ there is 
essentially at most four different values these assume.  
\end{abstract}

\section{{\large Introduction}}
A well studied model in quantum chaos is the so called quantized 
cat map - a ``quantized version'' of the 
dynamical system given by a hyperbolic 
(i.e. with $|\operatorname{tr}(A)|>2$) matrix $A\in SL(2,\inte )$ 
acting on the two dimensional torus. The quantization of these systems 
is a unitary operator $U_N(A)$ acting on the space 
$L^2\left(\inte _N\right)\cong \C ^N.$ 
This model was first introduced by Berry and Hannay \cite{Berry} and has been
developed in a number of papers \cite{Kna,Esp1,Esp2,Kli,Zel,KR2,Gur,Mez}.
The general idea is that the chaotic behavior of the classical system 
corresponds to eigenfunctions of the quantized system being 
``nicely spread out'' in the so called 
semiclassical limit, that is, when $N$ goes 
to infinity. $U_N(A)$ can have large degeneracies, but as Kurlberg and 
Rudnick explained in \cite{KR2}, this is a consequence of 
quantum symmetries in our model. Namely, there is a large abelian group of 
unitary 
operators commuting with $U_N(A).$ In analogy with the theory of modular forms, 
these operators are called Hecke operators and their joint eigenfunctions are 
called Hecke eigenfunctions. 
Kurlberg and Rudnick showed that the Hecke eigenfunctions 
become uniformly distributed as $N\to\infty,$ a fact often referred 
to as arithmetic quantum 
unique ergodicity (QUE) for cat maps. 

Another natural question relating to
eigenfunctions ``spreading out'' in the limit is the question of 
estimating their supremum norms. Given the matrix $A,$ the primes 
(all but a finite number of them to be exact) 
can be divided in two parts called split and inert, and  
in \cite{KR} and \cite{Kurl2} it was shown that for such 
prime numbers $N$ the 
supremum norm of $L^2-$normalized Hecke 
eigenfunctions are bounded by $2/\sqrt{1-1/N}$ and $2/\sqrt{1+1/N}$ 
respectively. As an immediate consequence of 
this it follows that as long as $N$ is square free, all Hecke eigenfunctions 
$\psi$ fulfill $\| \psi \|_\infty=O\left(N^\epsilon\right)$ 
for all $\epsilon >0.$
For general $N$ we 
only know that the supremum is 
$O_\epsilon\left(N^{3/8+\epsilon}\right)$ for all $\epsilon >0$ (cf. \cite{KR}).
In view of the results for prime numbers $N$ and the quantum unique ergodicity, 
one might think that all Hecke eigenfunctions 
have small supremum norm, maybe even smaller than $N^\epsilon$ for all 
$\epsilon>0,$ however this is not the case. In this paper we observe that, 
unless $N$ is square free, some of the 
Hecke eigenfunctions are localized on ideals of $\inte _N$ and for such 
functions we get rather large supremum norms. To be more precise, if 
$N=a^2$ we can find an eigenfunction with supremum norm $\gg N^{1/4}.$
This result is
somewhat analogous with the result of Rudnick and Sarnak \cite{Sar} 
concerning the 
supremum norm of eigenfunctions of the Laplacian of a special class of 
arithmetic compact 3-manifolds. They show that the supremum of some so called 
``theta lifts'' are $\gg\lambda^{1/4},$ where $\lambda$ is the 
corresponding eigenvalue. For a $L^2$-normalized 
function in $L^2\left(\inte _N\right)$ it is trivial to see that the maximal 
supremum is $N^{1/2}$ and the (sharp) general upper bound 
for the supremum of an eigenfunction of the Laplacian of a compact manifold is 
$O\left(\lambda^{(d-1)/4}\right),$ where $d$ is the dimension of the manifold
(cf. \cite{See}). For $d=3,$ we see that the growth we obtain for our 
eigenfunctions is analogous to  
the growth of the ``theta lifts''
in the sense that they are both the square root of the largest possible 
growth.

In Theorem~\ref{iso} we note that 
the action of $U_N$ on the subspace spanned by the Hecke eigenfunctions 
localized on ideal is isomorphic to
the action of $U_{N'}$ on $L^2\left(\inte _{N'}\right)$ for some $N'|N.$ 
This means that one can think of these eigenfunctions as the analogue of 
what in the theory of automorphic forms is called old forms.
Hecke eigenfunctions that are orthogonal to the old forms play the role of 
new forms. Note that the existence of old forms, although their supremum 
is large, has small relation to the concept of scaring. On the one hand we 
know from the result of Kurlberg and Rudnick that no scaring is possible for 
Hecke eigenfunctions, 
and on the other hand the ideals themselves equidistribute, 
hence it is not surprising that old forms do not contribute to scars. 

Another quantity one can study in order to determine how well 
eigenfunctions ``spread out'' is the Shannon entropy, a large entropy signifies 
a well-spread function. This has been done in a recent 
paper by Anantharaman and Nonnenmacher \cite{Ana} 
for the baker's map. In this study they use estimates from below of the Shannon 
entropy to show that the Kolmogorov-Sinai entropy of the 
induced limit measures (so called semiclassical measures) is always at least 
half of the Kolmogorov-Sinai entropy of the Lebesgue measure. 
We prove that the equivalent estimate
of the Shannon entropies is true for eigenfunctions of the quantized cat map and 
that our large eigenfunctions makes this estimate sharp. Even though the 
Hecke eigenfunctions do not contribute to scars 
(which other sequences of eigenfunctions do) they are still as badly spread out 
as possible in the sense of entropy.
That is, even though the only limiting measure of Hecke functions is the 
Lebesgue measure and this has {\em maximal} Kolmogorov-Sinai entropy, 
some of the Hecke functions have {\em minimal} Shannon entropy.

\begin{figure}[!h]
\begin{center}
\begin{tabular}{c c}
\includegraphics [width=5.9 cm]{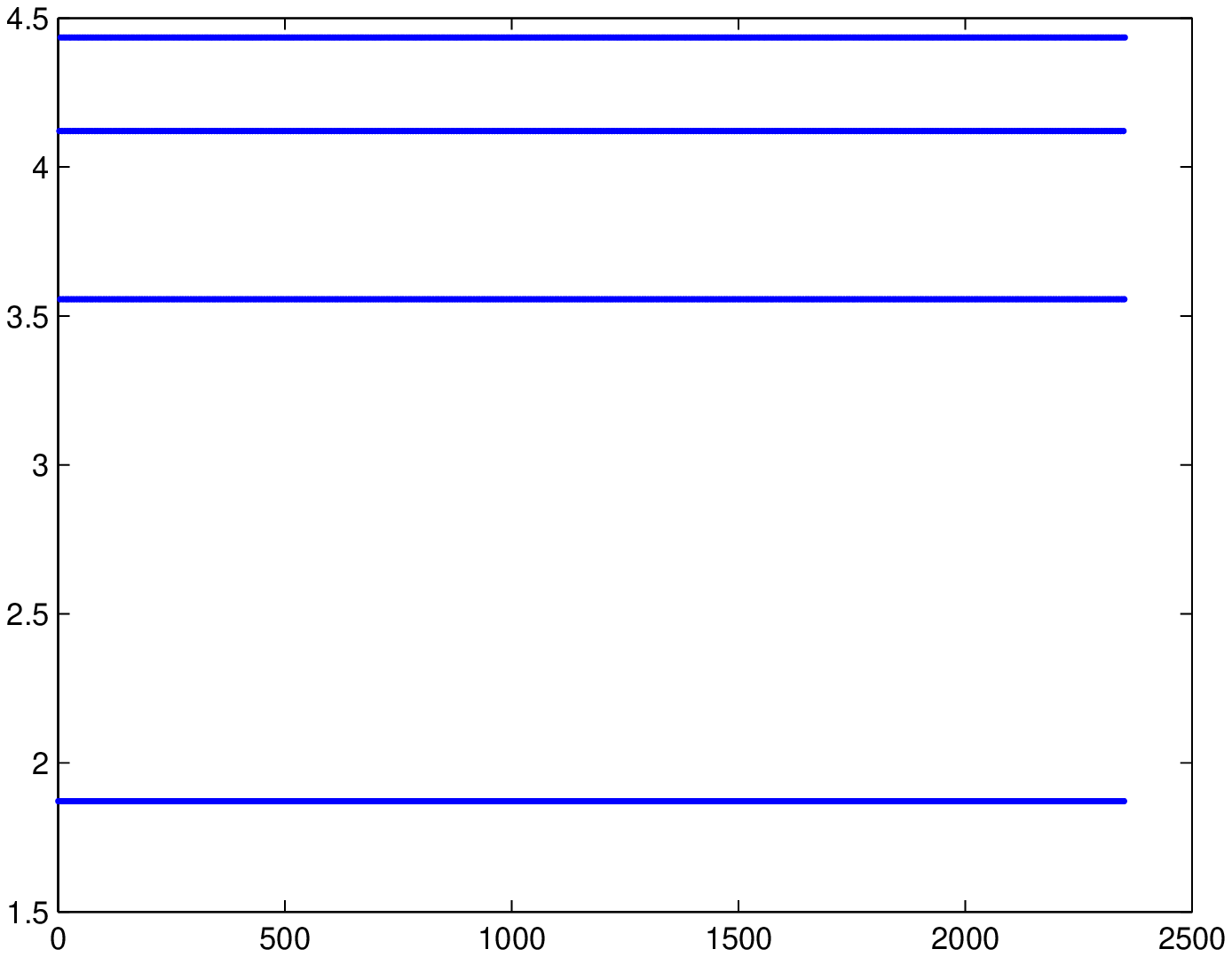} 
& \includegraphics [width=5.9 cm]{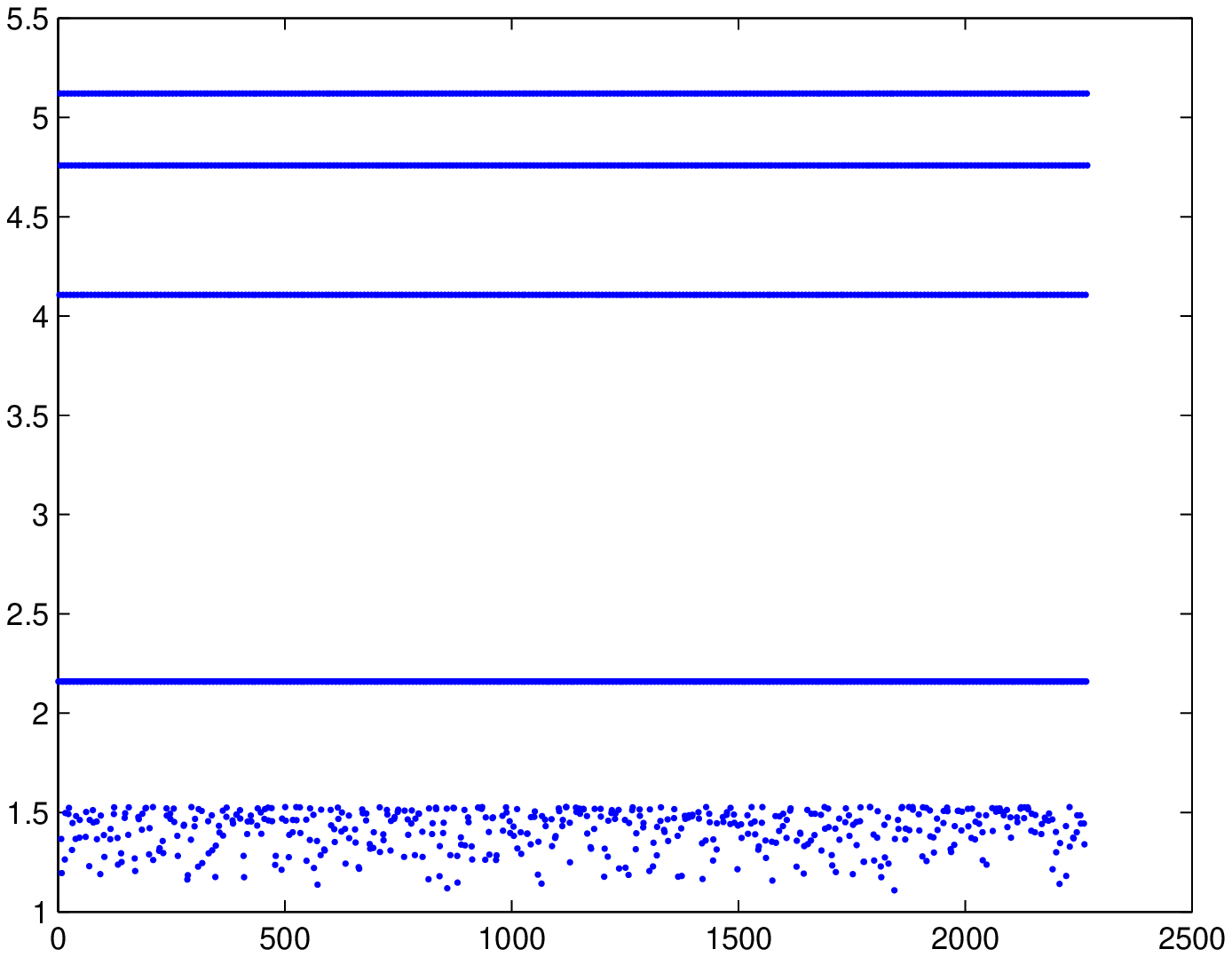}\\
\includegraphics [width=5.9 cm]{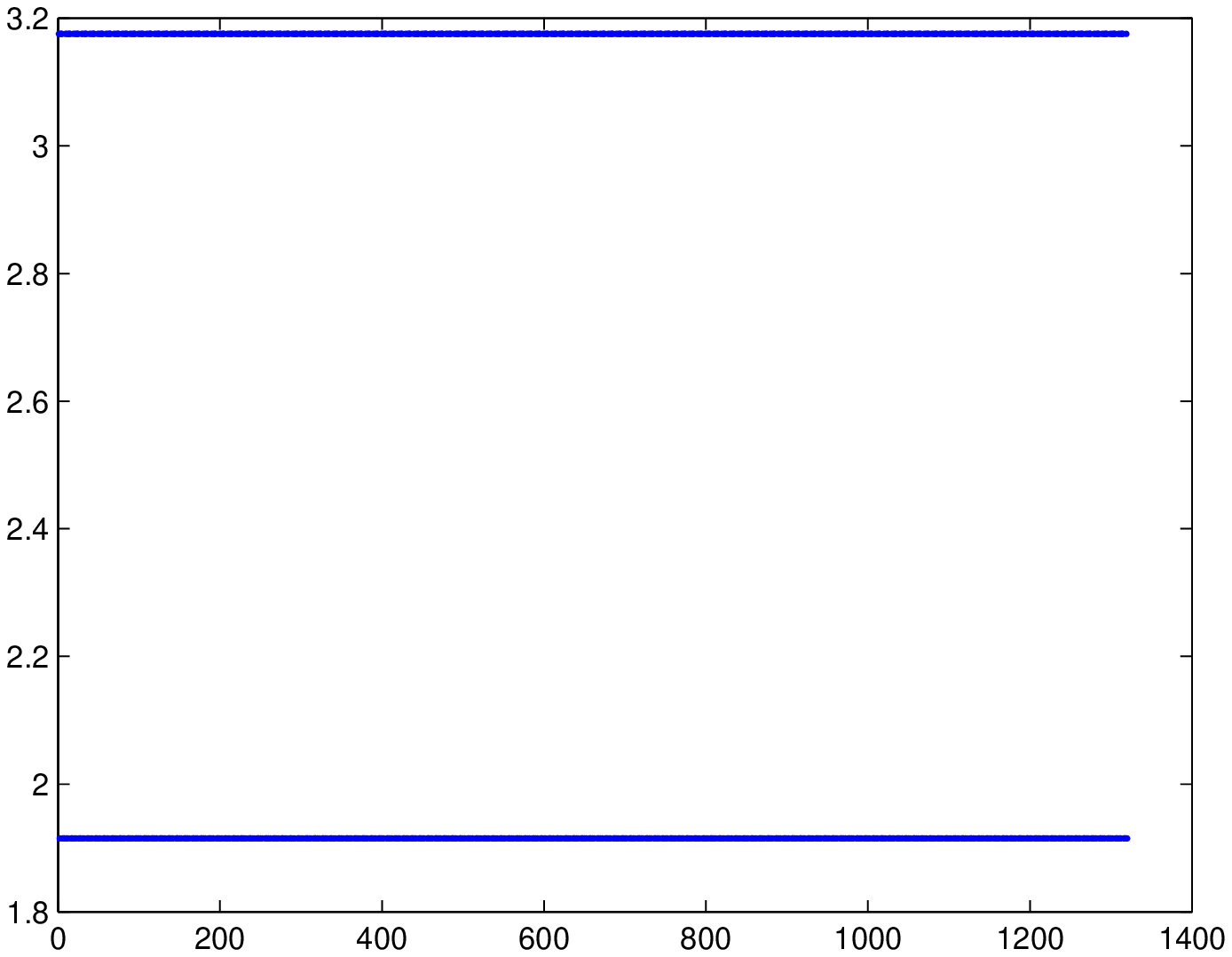} 
& \includegraphics [width=5.9 cm]{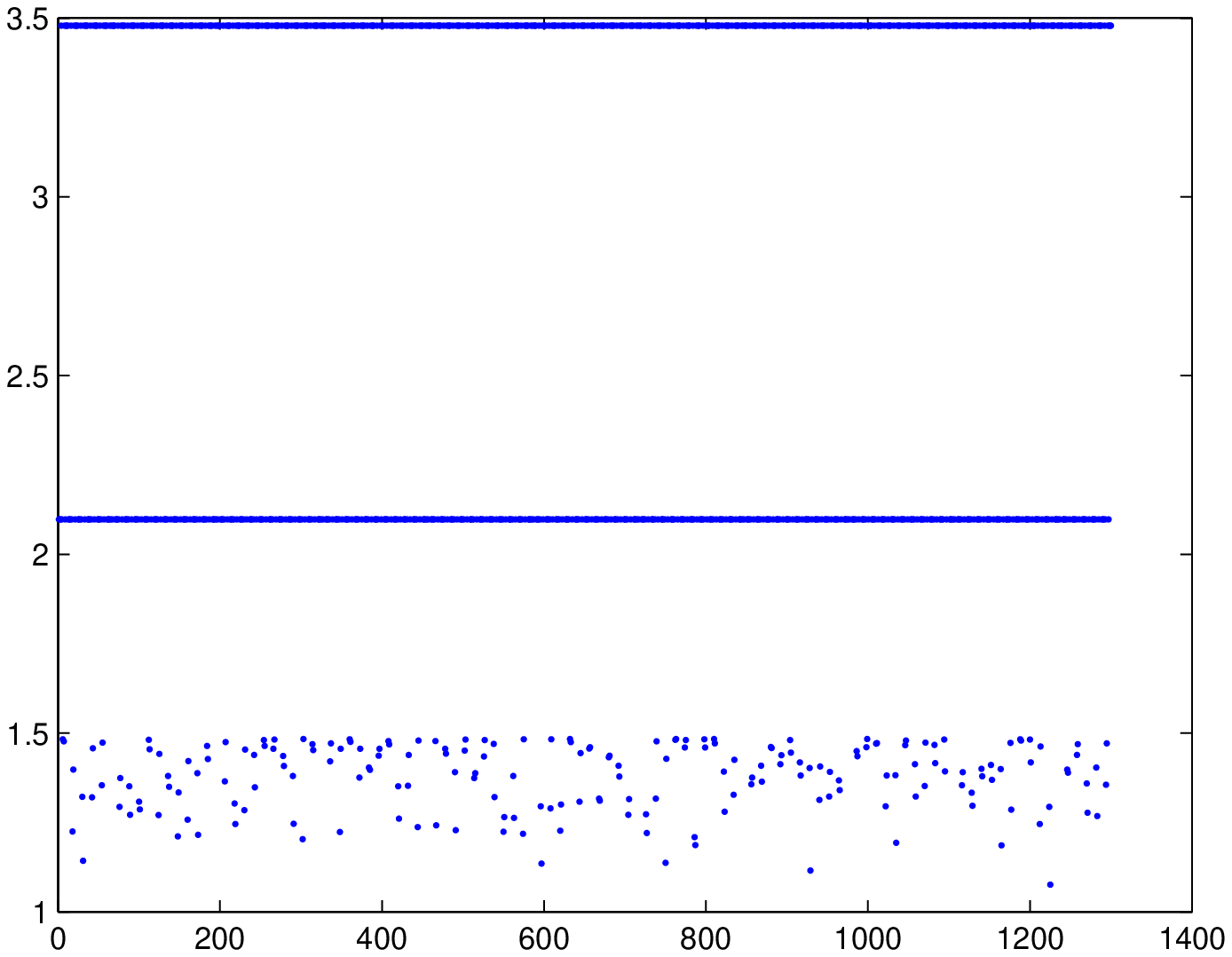}
\end{tabular}
\caption{The supremum norm of all the new forms of a given matrix $A$: 
in the upper row $N=7^4$ and in the lower $N=11^3$
and in the left 
column the primes ( i.e. $7,11$) are inert, while 
in the right column, they are split. The new forms are ordered with respect to
growing phase (in the interval $[-\pi,\pi)$) 
of their eigenvalues, when these are 
evaluated at some specific element of maximal order in the Hecke algebra.}
\end{center}
\end{figure}

In the study of new forms a very surprising phenomena occurs;
assume for simplicity that $N=p^k$ with $k>1,$ then it seems like the 
space is divided into two or four different 
subspaces and Hecke eigenfunctions in the same space have the same or almost 
the same supremum norm. 
These norms are not dependent on $A$ other than that the normalization 
factor is different if $A$ makes $p$ split or inert. We will derive these 
properties in the case where the power of $p$ is even. This is done using an 
arithmetic description of the Hecke eigenfunctions introducing two parameters
$C$ and $D$ where the different lines corresponds to the solvability of second 
and 
third order equations of $CD$ modulo $p.$ Moreover, the exact values these 
supremum norms are calculated. The lower line is not a true line but rather a 
strip of width $O(N^{-1})$ below the value $2/\sqrt{1\pm 1/p}$
corresponding to $p$ being split or inert. This is the same value as 
the known bound for primes $N.$ 
The other lines are true lines and their value is calculated in 
Theorem~\ref{supformel}, the values are of the size $N^{1/6}.$ 
The ``noise'' we see for the split case is also explained and corresponds to 
$p|C.$  
But as we see in figure 1, numerical simulations 
suggests that similar properties hold also for odd powers and this will be 
explored in a forthcoming paper. 


Our calculations show that if $N=p^{2k}$ 
($p>3$ and $p$ is either split or inert) 
then the supremum of all Hecke eigenfunctions is bounded by $N^{1/4}$ and this 
estimate is sharp. By multiplicativity this is then true for all products of 
such $N.$

\section{{\large Acknowledgments}}
I would like to express me gratitude to my advisor Pär Kurlberg for his 
knowledge and enthusiasm for the subject. 
I also thank Michael Björklund for the many helpful discussions we have had 
concerning this paper.

\section{{\large Short description of the model}}
This will be a very brief introduction to quantized cat maps, more or less just 
introducing the notation we will use. A more extensive 
setup can be found in \cite{KR2}. 
\\[15 pt]
Take a matrix $A\in SL(2, \inte ).$ We assume that 
$|\operatorname{tr} (A)|>2$ to make the system chaotic and that 
the diagonal entries of $A$ are odd and the off diagonal are even. If $N$ is 
even we make the further assumption that $A\equiv I~(\mbox{mod}~4).$ These 
congruence assumptions are needed in order for the quantization of the 
dynamics to be consistent with the quantization of observables. 
In each time step we map 
$x\in \mathbb{T}^2\cong\mathbb{R}^2/\mathbb{Z}^2$ to $Ax\in \mathbb{T}^2$ 
and observables $f\in C^\infty\left( \mathbb{T}^2\right)$ are sent to 
$f\circ A.$ The quantization of this is a unitary operator $U_N(A)$ acting on 
``the state space'' $L^2\left(\inte _N\right),$ equipped with the inner product
$$<\phi,\psi>=\frac{1}{N}\sum_{Q\in \inte _N}\phi(Q)\overline{\psi(Q)}.$$
Assume for a moment that we know how to define $U_N(A)$ when $N$ is a prime 
power. For general $N$ we write 
$N=p_1^{\alpha_1}...p_m^{\alpha_m}$ and via the Chinese remainder theorem we 
get an isomorphism between $L^2\left(\inte _N\right)$ and 
$\bigotimes_{j=1}^mL^2\left(\inte _{p_j^{\alpha_j}}\right).$ Using this 
decomposition we define $U_N(A):=\otimes_{j=1}^mU_{p_j^{\alpha_j}}(A).$
We now only have to define $U_{p^k}(A):$
Identify $A$ with its image in $SL(2,\inte_{p^k})$ 
and use the Weil representation to quantize $A.$ 
This is a representation of $SL(2,\inte_{p^k})$ on 
$L^2\left(\inte _{p^k}\right),$
which for odd primes $p$ is
given on the generators by
\begin{align}
\label{1}
U_{p^k}(n_b)\psi(x)&=e\left(\frac{rbx^2}{p^k}\right)\psi(x)\\
\label{2}
U_{p^k}(a_t)\psi(x)&=\Lambda(t)\psi(tx)\\
\label{3}
U_{p^k}(\omega)\psi(x)&=\frac{S_r\left(-1,p^k\right)}{\sqrt{p^k}}\sum_
{y\in\inte _{p^k}}\psi(y)e\left(\frac{2rxy}{p^k}\right),
\end{align}
where we have introduced the notation 
$$n_b= \left( \begin {array}{ccc} 1 & b\\
\noalign{\medskip} 0 & 1\end {array} \right),~~
a_t= \left( \begin {array}{ccc} t & 0\\
\noalign{\medskip} 0 & t^{-1}\end {array} \right),~~
\omega=\left( \begin {array}{ccc} 0 & 1\\
\noalign{\medskip} -1 & 0\end {array} \right) \mbox{~and~~}e(x)=e^{i2\pi x}.$$ 
($\Lambda(t)$ and $S_r\left(-1,p^k\right)$ are numbers with absolute value 1 and 
$r$ is a specific unit in $\inte_N,$ see \cite{KR2} for details.)
For $p=2$ the construction is similar but not quite the same: 
The subgroup of all matrices congruent to the identity modulo $4$ in 
$SL\left(2,\inte _{2^k}\right)$ 
is generated by $n_b,$ $n_c^T,$ $a_t$ and 
$U_{2^k}(n_b),$ $U_{2^k}(a_t)$ are given by (\ref{1}) and (\ref{2}) with $p=2.$
Finally $U_{2^k}(n_c^T)=H^{-1}U_{2^k}(n_{-c})H,$ where $H$ is the operator 
associated to the expression in (\ref{3}) with $p=2.$ 

The Hecke operators corresponding to the matrix $A$ 
are all the operators written as $U_N(g),$ where $g=xI+yA$ and $g$ has 
determinant $1.$

\section{{\large Hecke eigenfunctions with large supremum norm}}
\begin{define}
For $k\ge m\ge n$ we let 
$$S_k(m,n)=\left\{f\in L^2\left(\mathbb{Z}_{p^k}\right); p^m |x-y 
\Rightarrow f(x)=f(y)~~~ \wedge~~~p^n \notdivides x \Rightarrow 
f(x)=0\right\}.$$
\end{define}
\begin{remark*}
$S_k(m,n)$ can be canonically embedded into $L^2\left(\mathbb{Q}_p\right).$ 
As functions of the p-adic numbers these functions are called Schwartz functions 
because of their analogy with the Schwartz functions of a real variable.
\end{remark*}
\begin{thm}
\label{inv}
Let $p$ be an odd prime and $m\le k\le 2m.$ 
Then $S_k(m,k-m)$ is invariant under the action of $U_{p^k}.$
\end{thm}
\begin{proof}
Let $f\in S_k(m,k-m).$ It is easy to see that 
$U_{p^k} (a_t)f\in S_k(m,k-m)$ and that $U_{p^k} (n_b)f(x)=0$ 
if $p^{k-m}\notdivides x.$
Moreover we have
\begin{align*}
U_{p^k}(n_b) f\left(p^{k-m}x+yp^m\right) =&
e\left(\frac{rb\left(p^{k-m}\left(x+yp^{2m-k}\right)\right)^2}{p^k}\right)
f\left(p^{k-m}x+yp^m\right)\\
=&e\left(\frac{rb\left(x+yp^{2m-k}\right)^2}{p^{2m-k}}\right)f\left(p^{k-m}
x\right)\\
=&e\left(\frac{rbx^2}{p^{2m-k}}\right)f\left(p^{k-m}
x\right)
=U_{p^k}(n_b)f\left(p^{k-m}x\right)
\end{align*}
and
\begin{align*}
U_{p^k}(\omega) f(x) =& 
\frac{S_r\left(-1,p^k \right)}{\sqrt{p^k}}\sum_{y\in \inte_{p^k}}
f(y)e\left(\frac{2rxy}{p^k}\right)\\
=&\frac{S_r\left(-1,p^k \right)}{\sqrt{p^k}}\sum_{y\in \inte_{p^m}}
f\left(yp^{k-m}\right)e\left(\frac{2rxy}{p^m}\right)\\
=&\frac{S_r\left(-1,p^k \right)}{\sqrt{p^k}}
\sum_{a\in \inte_{p^{2m-k}}}\sum_{b\in \inte_{p^{k-m}}}
f\left(\left(a+p^{2m-k}b\right)p^{k-m}\right)\\
&e\left(\frac{2rx\left(a+p^{2m-k}b\right)}{p^m}\right)\\
=&\frac{S_r\left(-1,p^k \right)}{\sqrt{p^k}}
\sum_{a\in \inte_{p^{2m-k}}}f\left(ap^{k-m}\right)e\left(\frac{2rxa}{p^m}\right)
\sum_{b\in \inte_{p^{k-m}}}e\left(\frac{2rxb}{p^{k-m}}\right).
\end{align*}
If $p^{k-m}\notdivides  x$ then the sum over $b$ is equal to zero, 
and the sum over 
$a$ only depends on the remainder of $x$ modulo $p^m.$ Thus 
$U_{p^k}(\omega) f\in S_k(m,k-m)$ which concludes the proof.
\end{proof}
\begin{thm}
\label{SATSEN}
Let $N=p^k$, where $p$ is an odd prime. 
Then there exists Hecke eigenfunctions $\psi\in L^2(\inte _N)$ 
such that $\| \psi \|_2=1$ and 
$$\| \psi \|_\infty \ge p^{\left[\frac{k}{2}\right]/2}.$$ 
\end{thm}
\begin{proof}
The Hecke operators are of the form $U_{p^k}(B)$ for $B\in SL(2,\inte)$
where all $B$ commute. 
Since $S_k(k-[k/2],[k/2])$ is $SL(2,\inte )$-invariant there must be a 
joint eigenfunction $\psi$ of all $U_{p^k}(B)$ such that 
$\psi\in S_k(k-[k/2],[k/2]).$ 
If this function is normalized to have  $\| \psi \|_2=1,$ we get that 
$$\frac{p^{k-[k/2]}}{p^k}\| \psi \|_\infty^2 \ge\| \psi \|_2^2= 1$$
by the estimation $|\psi(x)|\le \| \psi \|_\infty$ on the support of $\psi.$  
\end{proof}

\begin{remark*}
When $k$ is even the space $S_k(k-[k/2],[k/2])=S_k(k/2,k/2)=\complex f$ where 
\begin{equation}
\label{storfunktion}
f(x)=\left\{  \begin {array}{ccc} 1 & \textrm{~~~~~~if $p^{k/2}|x$}\\
0 & \textrm{else} \end {array}\right.
\end{equation}
and we have $U_{p^k}(A)f=f$ for all $A\in SL(2, \inte ).$
\end{remark*}
The action of the Weil representation on $S_k(k-m,m)$ is isomorphic to the 
action on the full space, but for a different $N.$ More precisely,
let 
$T_m:S_k(k-m,m)\rightarrow L^2\left(\inte_{p^{k-2m}}\right)$ be defined by  
$(T_m\psi)(x)=p^{-m/2}\psi(p^mx),$ then $T_m$ is a bijective intertwining 
operator. In other words:
\begin{thm}
\label{iso}
Let $N=p^k,$ where $p$ is an odd prime. The operators $T_m$ are bijective and if 
$\psi\in S_k(k-m,m)$ we have that 
$U_{p^k}(A)\psi=T_m^{-1}U_{p^{k-2m}}(A)T_m\psi.$
\end{thm}
\begin{proof}
That $T_m$ is well defined and bijective is trivial. We are left with proving 
that the identity holds for the generators of $SL(2,\inte_{p^k}).$ This is 
immediate for $n_b$ and $a_t$ and for $\omega$ we have
\begin{align*}
(T_mU_{p^k}(\omega)\psi)(x)&=\frac{S_r\left(-1,p^k\right)}{\sqrt{p^{k+m}}}\sum_
{y\in\inte _{p^k}}\psi(y)e\left(\frac{2rxy}{p^{k-m}}\right)\\
&=\frac{S_r\left(-1,p^{k-2m}\right)}{\sqrt{p^{k-m}}}\sum_
{y\in\inte _{p^{k-m}}}\psi(y)e\left(\frac{2rxy}{p^{k-m}}\right)\\
&=\frac{S_r\left(-1,p^{k-2m}\right)}{\sqrt{p^{k-2m}}}\sum_
{y\in\inte _{p^{k-2m}}}p^{-m/2}\psi(p^my)e\left(\frac{2rxy}{p^{k-2m}}\right)\\
&=(U_{p^{k-2m}}(\omega)T_m\psi)(x).
\end{align*}
\end{proof}
\begin{remark*}
$T_m$ is in fact unitary.
\end{remark*}

One can obtain results analogous to Theorem~\ref{inv} and Theorem~\ref{SATSEN} 
for $p=2.$ 
\begin{thm}
\label{SATSEN2}
Let $p=2$ and $m\le k\le 2m+1.$ Then $S_k(m,k-1-m)$ is invariant under the 
action of $U_{2^k}.$
\end{thm}
\begin{proof}
Observe that we only need to show that $S_k(m,k-1-m)$ is preserved by (\ref{1}), 
(\ref{2}) and (\ref{3}) and do the same calculations as in the proof of 
Theorem~\ref{inv}.
\end{proof}
\begin{cor}
Assume that $N=ab^2,$ where $b$ is odd, or that $N=2ab^2$. 
Then, in both situations, there exists normalized Hecke eigenfunctions 
$\psi\in L^2\left(\inte _N\right)$ such that 
$$\| \psi \|_\infty \ge b^{1/2}.$$ 
\end{cor}
\begin{proof}
Follows immediately from Theorem \ref{SATSEN} and Theorem \ref{SATSEN2} since 
$\|f\otimes g\|_\infty=\|f\|_\infty\|g\|_\infty.$
\end{proof}

\section{{\large Shannon entropies of Hecke functions}}

Entropy is a classical measure of uncertainty (chaos) in a dynamical system and 
recently this has been studied in a number of papers in the context of 
quantum chaos, see \cite{Ana},\cite{Ana2}. 
The main conjecture can intuitively be described in the following way: 
The entropy is always at least half of the largest possible entropy. 

\begin{define}
Let $f\in L^2\left(\inte_N\right)$ and assume $\|f\|_2=1.$ We define the 
Shannon entropy to be 
$$h(f)=-\sum_{x\in\inte_N}\frac{|f(x)|^2}{N}\log{\frac{|f(x)|^2}{N}}.$$
\end{define}

In \cite{Ana}, Anantharaman and Nonnenmacher prove the described 
conjecture in the case of semiclassical limits of the Walsh-quantized baker's 
map with $N=D^k$ and $D$ fixed. 
In the course of this proof they come across similar inequalities for 
the Shannon entropy of the specific eigenstates. The maximal entropy is 
$|\log {2\pi \hslash }|$ (where $\hslash$ is Planck's constant) and they show 
that each eigenstate $\psi_\hslash$ fulfills 
$h(\psi_\hslash )\ge 1/2 |\log {2\pi \hslash }|$. In our case $2\pi \hslash$ is 
equivalent to $N^{-1}$ and therefore a natural question for cat maps is if
\begin{equation}
\label{ent}
h(\psi)\ge \frac{1}{2}\log{N}
\end{equation}
for Hecke eigenfunctions, or more generally, for eigenfunctions of $U_N(A).$
Let us first note that if we for instance take $N$ to be prime and put $A=n_b$ 
for some $b\varnotsign= 0$ then the function 
$$f(x)=\left\{  \begin {array}{ccc} \sqrt{N} & \textrm{~~~~~~if $x=0$}\\
0 & \textrm{else} \end {array}\right.$$
fulfill $U_N(A)f=f$ and $h(f)=0,$ hence the inequality in (\ref{ent}) 
can not be true in full generality. However the following is true:

\begin{thm}
\label{entropi}
Assume that $A$ is not upper triangular modulo $p$ for any $p|N.$ 
If $f\in L^2\left(\inte_N\right)$ is a normalized eigenfunction of $U_N(A)$ 
then $h(f)\ge \frac{1}{2}\log{N}.$  
\end{thm}

If $f$ is the function defined in (\ref{storfunktion}) then $N^{1/4}f=p^{k/4}f$ 
fulfills $h(f)=1/2\log N,$ hence the inequality in Theorem~\ref{entropi} 
is sharp. The proof of the Theorem is a simple application of the following 
Entropic Uncertainty Principle which can be found in \cite{Ana}:

\begin{thm}{\bf Entropic Uncertainty Principle }
Let $N$ be a positive integer and let $U$ be a unitary $N\times N$ matrix. 
If we denote $c(U)= \max |U_{i,j}|,$ then 
$$h(f)+h(Uf)\ge -2\log{c(U)}$$
for all $f\in L^2 \left(\inte_N\right)$ with $\|f\|_2=1.$
\end{thm}

\begin{proof}[Proof of Theorem \ref{entropi}]
It is enough to prove the statement for $N=p^k.$ 
Let $A=\left(\begin {array}{ccc} a & b\\
\noalign{\medskip} c & d\end {array} \right).$ Since $p\notdivides c$ we can 
write $A=n_{b_1}\omega n_{b_2} a_t$ where $t=-c,$ $b_1=ac^{-1}$ and $b_2=cd.$
Inserting this into the definition of $U_N$ we get 
$$U_N(A)\psi(x)=\frac{S_r(-1,N)}{\sqrt{N}}\Lambda(t)\sum_{y\in\inte_N}
e\left(\frac{r\left( b_1x^2+b_2y^2+2xy)\right)}{N} \right)\psi(ty).$$
Hence, if we view $U_N(A)$ as an $N\times N$ matrix then all its entries have 
absolute value $N^{-1/2}$ and thus if $U_N(A)f=\lambda f$ then the 
Entropic Uncertainty Principle says that
$h(f)=h(U_N(A)f)\ge -\log N^{-1/2}=1/2\log N.$  
\end{proof}

Note that the function defined in (\ref{storfunktion}) is invariant under the 
action of $SL(2,\inte_N)$ and in particular if we apply the Fourier transform to 
it. Thus the Shannon entropy of the state is trivially the same in both the
``position''-representation and the ``momentum''-representation. This property 
was also observed by Anantharaman and Nonnenmacher in their study of the 
baker's map, however there is a big difference between the two quantizations. 
The baker's map is quantized in a manner where different states correspond to 
different functions of phase space (the torus) and in this context it is natural
to study the so called Wehrl entropy of the state \cite{Weh}.
They prove that the Wehrl 
entropy coincides with the Shannon entropy. Our states are elements in 
the state space $L^2(\inte_N)$ and in our quantization 
$g \mapsto <Op(g)\psi,\psi>$ (see \cite{KR2} for the definition of $Op$) 
is a signed measure, but does not induce a density on the phase space. 
If we in particular pick $\psi$ to be the function in (\ref{storfunktion}), then 
for all $g\in C^\infty (\mathbb{T}^2)$ we have that
$$<Op(g)\psi,\psi>=\int_{\mathbb{T}^2}w(x)g(x)dx,$$
where $w(x)$ is the $p^{-k/2}-$periodic extension of 
$$w(x)=\frac{1}{2p^k}\left(\delta_{0,0}+\delta_{p^{-k/2}/2,0}
+\delta_{0,p^{-k/2}/2}-\delta_{p^{-k/2}/2,p^{-k/2}/2}\right).$$
This can be seen using Poisson's summation formula or by straightforward 
identification of Fourier coefficients.  
There is however a problem with the large class $C^\infty (\mathbb{T}^2)$ 
of observables (we want to think of $w(x)$ as a positive function, 
but obviously it is not) 
seen also from a physicists perspective; 
the Heisenberg uncertainty principle says that we cannot 
measure exact points (i.e exact position and exact momentum) 
in phase space, thus our observables should not 
behave to badly on a local scale. One naive way to cope with this problem would 
be to study trigonometric polynomials as observables and let the number of
terms in the trigonometric expansion to grow with $N.$ This solves our problems 
and makes it possible for us to approximate our sum of delta functions 
by a trigonometric polynomial.
To be more precise: Let $\Omega\subset \inte^2$ be a bounded set and let
$T(\Omega)=\left\{f(x)=\sum_{n\in\Omega}c_ne(n\cdot x);c_n\in\C\right\}.$ 
Then for $\psi$ given by (\ref{storfunktion}) we have
$$<Op(g)\psi,\psi>=\int_{\mathbb{T}^2}\tilde{w}(x)g(x)dx,$$
for $g\in T(\Omega)$ and with 
$\tilde{w}(x)=\sum_{n\in \Omega\cap p^{k/2}\inte^2} 
(-1)^{n_1n_2}e(n\cdot x).$ 
In particular if we let $\Omega$ be a disc of radius square root of the 
inverse of Planck's constant, i.e. $\Omega=\{x\in \inte^2; |x|<p^{k/2}\},$ 
we have $\tilde{w}(x)=1.$ Note that the Wehrl entropy of $\tilde{w}(x)=1$ 
is maximal ($\log{p^k}$), but that the Shannon entropy of $\psi$ is minimal.

\section{{\large Evaluation of Hecke eigenfunctions}}
The rest of the paper is devoted to the study of Hecke eigenfunctions in the 
orthogonal complement of $S_k(k-1,1).$ In view of Theorem~\ref{iso} this is no 
restriction, but rather a natural way to structure the theory.
To get an easy description of the 
dynamics we will make the assumption that $N=p^{2k}$ where $p$ is a prime larger 
than $3.$ The fact that the
dynamics seems to be easier to describe if $N$ is assumed to be a perfect square,
has been observed before by Knabe \cite{Kna}. 
Although his quantization is different, the description of the dynamics 
is quite similar. We begin the study by some basic definitions:

\begin{define}
For $x\in \inte_N,$ let $\delta_x:\inte_N\rightarrow \C$ be the function which is
$1$ at $x$ and $0$ at every other point.  
\end{define}

\begin{define}
Given $x=\left(\begin {array}{c} x_1\\x_2\end {array}\right)\in\inte_N^2,$ 
let $\zeta_x:\inte_N\rightarrow \C$ be defined by
$$\zeta_x=\sum_{t\in\inte_{p^k}}e\left(\frac{x_1t}{p^k}\right)
\delta_{x_2+p^kt}.$$
\end{define}

\begin{remark*}
Notice that $\{\zeta_x;x\in\{1,2,...,p^k\}^2 \}$ is an orthogonal base of 
$L^2(\inte_N)$ and that $x\equiv x'~\pmod {p^k}$ implies $\zeta_x=c\zeta_{x'}$ 
for some number $c$ such that $c^{p^k}=1.$ In particular the space $\C \zeta_x$ 
can be thought of as defined for $x\in \inte_{p^k}^2.$
\end{remark*}

\begin{define}
Given $D\in \inte_N$ we let $$H_D=\left\{ \left(\begin {array}{ccc} a & bD\\
\noalign{\medskip} b & a\end {array} \right) ; 
a,b\in\inte_N~,~a^2-Db^2=1 \right\}.$$
\end{define}
We make the assumption that 
$A$ is not upper triangular modulo $p.$ Because of this assumption 
$A$ can be written as
$A=n_{b}hn_{-b}$ for some $b,D$ and some $h\in H_D$ and so we see that 
the Hecke operators are given by $\{ U_N(g);g\in n_{b}H_Dn_{-b}\}.$
But if $\psi$ is an eigenfunction of $U_N(h),$ then 
$\widetilde\psi=U_N(n_b)\psi$ is an eigenfunction of $U_N(n_bhn_{-b})$ and 
furthermore $|\psi(x)|=|\widetilde\psi(x)|.$ Thus
we may assume that the Hecke operators are $\{ U_N(h);h\in H_D\}.$ If $D$ is a 
quadratic residue modulo $p$ then $p$ is called split, if $D$ is not a 
quadratic residue modulo $p$ then $p$ is called inert, and if $p|D$ then $p$ is 
called ramified.

\begin{define}
Let $\mathscr{N}: \inte_{p^k}^2\rightarrow \inte_{p^k}$ be defined by 
$\mathscr{N}(x)=x_1^2-Dx_2^2.$ 
\end{define}

\begin{define}
For $C\in\inte_{p^k}$ we define 
$$V_C=\bigoplus_{\begin{subarray}{c} x\in \inte_{p^k}^2\\ 
\mathscr{N}(x)=-C\end{subarray}}\C\zeta_x.$$
\end{define}
\begin{remark*}
Note that $S_{2k}(2k-1,1)=\bigoplus_{p|x}\C \zeta_x\subseteq 
\bigoplus_{p|C}V_C$ and that the latter is an equality if $p$ is inert.
\end{remark*}

\begin{lemma}
\label{dynamik}
Assume $B\in SL(2,\inte_N)$ and that $x'=Bx.$ We have that 
$$U_N\left(B\right)\zeta_x
=e\left(\frac{r(x_1'x_2'-x_1x_2)}{N}\right)\zeta_{x'}.$$
\end{lemma}
\begin{proof}
By the multiplicativity of both sides of the equality it is enough to prove the 
lemma for the generators of $SL(2,\inte_N).$ Since $N=p^{2k}$ 
we have that $\Lambda(t)=S_r\left(-1,p^{2k}\right)=1$ and 
$2r\equiv 1~~(\mbox{mod}~N)$ (see \cite{KR2}). 
Using the definition of $U_N$ we get
\begin{align*}
U_N\left(n_b\right)\zeta_x&=
\sum_{t\in\inte_{p^k}}e\left(\frac{x_1t}{p^k}\right)
e\left(\frac{rb\left(x_2+p^kt \right)^2}{N}\right)\delta_{x_2+p^kt}\\
&=e\left(\frac{rbx_2^2}{N}\right)
\sum_{t\in\inte_{p^k}}e\left(\frac{(x_1+bx_2)t}{p^k}\right)\delta_{x_2+p^kt}
=e\left(\frac{r(x_1'x_2'-x_1x_2)}{N}\right)\zeta_{x'},\\
U_N\left(a_s\right)\zeta_x&=
\sum_{t\in\inte_{p^k}}e\left(\frac{x_1t}{p^k}\right)
\delta_{s^{-1}(x_2+p^kt)}
=\sum_{t\in\inte_{p^k}}e\left(\frac{sx_1t}{p^k}\right)
\delta_{s^{-1}x_2+p^kt}\\
&=e\left(\frac{r(x_1'x_2'-x_1x_2)}{N}\right)\zeta_{x'}
\end{align*}
and
\begin{align*}
U_N\left(\omega\right)\zeta_x 
&=\sum_{t\in\inte_{p^k}}e\left(\frac{x_1t}{p^k}\right)
\frac{1}{p^{k}}\sum_{z\in \inte_{N}}\delta_{x_2+p^kt}(z)
e\left(\frac{2ryz}{N}\right)\\
&=\frac{e\left(\frac{x_2y}{N}\right)}{p^{k}}\sum_{t\in\inte_{p^k}}
e\left(\frac{(x_1+y)t}{p^k}\right)
=e\left(\frac{x_2y}{N}\right)\sum_{t\in\inte_{p^k}}\delta_{-x_1+p^kt}\\
&=\sum_{t\in\inte_{p^k}}e\left(\frac{x_2(-x_1+p^kt)}{N}\right)\delta_{-x_1+p^kt}
=e\left(\frac{r(x_1'x_2'-x_1x_2)}{N}\right)\zeta_{x'}.
\end{align*}
\end{proof}

As a special case of Lemma~\ref{dynamik} we get the following corollary:
\begin{cor}
\label{egenfunktioner}
We have that
$$U_N\left(\begin {array}{ccc} 1 & tp^kD\\
\noalign{\medskip} tp^k & 1\end {array} \right)\zeta_x
=e\left(-\frac{r\mathscr{N}(x)t}{p^k}\right)\zeta_x.$$
\end{cor}
\begin{proof}
With $x'=Bx$ we have 
\begin{align*}
U_N&\left(\begin {array}{ccc} 1 & tp^kD\\
\noalign{\medskip} tp^k & 1\end {array} \right)\zeta_x=
e\left(\frac{r(x_1'x_2'-x_1x_2)}{N}\right)\zeta_{x'}\\
&=e\left(\frac{r\left(x_1^2+Dx_2^2\right)t}{p^k}\right)\sum_{s\in\inte_{p^k}}
e\left(\frac{\left(x_1+tp^kDx_2\right)s}{p^k}\right)\delta_{x_2+tp^kx_1+sp^k}\\
&=e\left(\frac{r\left(x_1^2+Dx_2^2\right)t}{p^k}\right)\sum_{s\in\inte_{p^k}}
e\left(\frac{x_1\left(s-tx_1\right)}{p^k}\right)\delta_{x_2+sp^k}
=e\left(-\frac{r\mathscr{N}(x)t}{p^k}\right)\zeta_x.
\end{align*}
\end{proof}
Since $H_D$ preserves $\mathscr{N}(x)$ we see that $V_C$ are natural subspaces 
to study. Pick $x_0\in \inte_{p^k}^2$ with $\mathscr{N}(x_0)=-C$ such that 
$p\notdivides C.$ We see that $H_Dx_0=\{x\in\inte_{p^k}^2;\mathscr{N}(x)=-C\},$ 
hence if a 
Hecke eigenfunction has a nonzero coefficient for some $\zeta_x\in V_C,$ then it 
has nonzero coefficients for all $\zeta_x\in V_C$ and they have the same 
absolute value. On the other hand, if $\psi$ is a Hecke eigenfunction, then 
Corollary~\ref{egenfunktioner} tells us that $\psi\in V_C$ for some 
$C\in\inte_{p^k}.$ 
If $p|C$ the orbit $H_Dx_0$ is not always as large. This corresponds to 
the fact that the irreducible representations in $V_C$ are no longer one 
dimensional.
\begin{lemma}
\label{dim}
If $p$ does not divide $C$ or $D$ then
$\dim(V_C)=p^k-\left(\frac{D}{p}\right)p^{k-1}.$
\end{lemma}
\begin{proof}
To calculate the dimension we first prove that we can find $x_1$ and $x_2$ such 
that $x_1^2\equiv -C+Dx_2^2~~(\mbox{mod}~p^k).$ 
This is done by induction on $k$ where 
each induction step use Newton-Raphson approximation, a method known in number 
theory as Hensel's lemma. For $k=1$ we have $(p+1)/2$
different squares, so both the left hand side and the right hand side assumes 
$(p+1)/2$ different values and by the pigeon hole principle we must have a 
solution to the equation. Now assume we have 
$x_1$ and $x_2$ such that $x_1^2\equiv -C+Dx_2^2~~(\mbox{mod}~p^{n-1}).$ 
At least one 
of $x_1$ and $x_2$ is not divisible by $p$ and we may assume that this is $x_1.$
Putting $\widetilde{x_1}=x_1-\left(x_1^2-Dx_2^2+C\right)/(2x_1)$ we see that 
$\widetilde{x_1}^2\equiv -C+Dx_2^2~~(\mbox{mod}~p^{n}).$
Let $B=\left(\begin {array}{ccc} x_1 & x_2D\\
\noalign{\medskip} x_2 & x_1\end {array} \right)$ have determinant congruent to
$-C$ modulo $p^k.$ We see that $V_C=U_N(B)V_1,$ thus every $V_C$ has the same 
dimension. 
We now count the number of $(x,y)\in \inte_{p^k}^2$ such that 
$x^2-Dy^2\equiv 0~~(\mbox{mod}~p):$ If 
$\left(\frac{D}{p}\right)=-1$ we immediately get 
$x\equiv y\equiv 0~~(\mbox{mod}~p)$ 
which gives $p^{2k-2}$ solutions. But if $\left(\frac{D}{p}\right)=1$ 
we also get 
the solutions $y\in \inte_{p^k}^\times$ and 
$x\equiv \pm \sqrt{D}y~~(\mbox{mod}~p),$ so in this case the total number of 
solutions is $p^{2k-2}+2p^{k-1}p^{k-1}(p-1)=(2p-1)p^{2k-2}.$
From this we see that for $\left(\frac{D}{p}\right)=-1$ we have
$$\dim(V_C)
=\frac{1}{p^{k-1}(p-1)}\dim\left(\bigoplus_{C\in\inte_{p^k}^\times}V_C\right)
=\frac{p^{2k}-p^{2k-2}}{p^{k-1}(p-1)}=p^k+p^{k-1}$$
and for $\left(\frac{D}{p}\right)=1$ we have
$$\dim(V_C)
=\frac{1}{p^{k-1}(p-1)}\dim\left(\bigoplus_{C\in\inte_{p^k}^\times}V_C\right)
=\frac{p^{2k}-(2p-1)p^{2k-2}}{p^{k-1}(p-1)}=p^k-p^{k-1}.$$
\end{proof}

The evaluation of a Hecke eigenfunction will lead to the study of the solutions 
to the equation $x^2\equiv a~~\left(\mbox{mod}~p^k\right).$ 
It is easy to see that if 
$a\not\equiv 0~~\left(\mbox{mod}~p^k\right)$ 
and $p$ divides $a$ an odd number of 
times, then the equation has no solutions. If however $p$ divides $a$ an even 
number of times we may reduce the equation to 
$\widetilde x^2\equiv \widetilde a~~\left(\mbox{mod}~p^{k-2s}\right),$ where 
$p\notdivides \widetilde a.$
If $\widetilde a$ is a square modulo $p$ then this equation has two solutions 
$\pm x_0$ and the solutions to the original equation are 
$x\equiv \pm x_0x^s +p^{k-s}m~~\left(\mbox{mod}~p^k\right)$ 
for $m\in \inte_{p^s}.$ 
If $a\equiv 0~~\left(\mbox{mod}~p^k\right)$ then the solutions are 
$x\equiv p^{\lceil k/2\rceil }m~~\left(\mbox{mod}~p^k\right)$ for 
$m\in \inte_{p^{[k/2]}}.$ 
Since the solutions to the equation are written in quite different forms we 
formulate the evaluation in two different theorems corresponding to different 
right hand sides of the equation. 

\begin{thm}
\label{Sats1}
Let $\psi\in V_C$ be a normalized Hecke eigenfunction and assume that 
$p$ does not divide $C$ or $D.$ Let $b\in \inte_N$ and assume that the equation 
$x^2\equiv -C+Db^2~~\left(\mathrm{mod}~p^k\right)$ has the solutions 
$x\equiv \pm x_0p^s+p^{k-s}\inte_{p^s}~\left(\mathrm{mod}~p^k\right)$ 
for some $x_0$ and $s$ such that 
$p\notdivides x_0$ and 
$0\le s<k/2.$ Then 
\begin{align}
\label{uttryck1}
\psi(b)=\frac{1}{\sqrt{1-\left(\frac{D}{p}\right)\frac{1}{p}}}
\left(\alpha_\psi (b)\sum_{z=1}^{p^s}e\left(\frac{q_+(z)}{p^s}\right) 
+\beta_\psi (b)\sum_{z=1}^{p^s}e\left(\frac{q_-(z)}{p^s}\right)\right),
\end{align}
where $q_\pm(z)=r\left(\Theta _\psi (b)z 
\pm x_0Dbz^2+p^{k-2s}3^{-1}D^2b^2z^3\right)$
and $|\alpha_\psi (b)|=|\beta_\psi (b)|=1.$
\end{thm}

\begin{remark*}
The function $\Theta_\psi (b),$ which takes values in $\inte_{p^s},$ will be 
specified in equation (\ref{theta}).
\end{remark*}

\begin{proof}
We know that $\psi$ is a linear combination of $\zeta_x$ such that 
$\mathscr{N}(x)=-C.$ 
Fixing $x_0,$ any such $x$ can be written as $hx_0$ for some $h\in H_D,$ 
hence it follows from 
Lemma~\ref{dynamik} that all constants in this linear combination have the 
same absolute value $R.$ 
The orthogonality of $\{\zeta_x;x\in\{1,2,...,p^k\}^2 \}$
and Lemma~\ref{dim} gives
$$1=\| \psi \|_2^2=
\left(p^k-\left(\frac{D}{p}\right)p^{k-1}\right)\frac{R^2}{p^k},$$
thus $R=\left(1-\left(\frac{D}{p}\right)\frac{1}{p}\right)^{-1/2}.$
Since $\zeta_x(b)=0$ unless $x_2\equiv b~~\left(\mbox{mod}~p^k\right)$ 
the value of $\psi(b)$ is only a sum 
over $x\in \inte_{p^k}^2$ such that 
$x_1^2\equiv -C+Db^2~~\left(\mbox{mod}~p^k\right)$ and 
$x_2\equiv b~~\left(\mbox{mod}~p^k\right).$
By the assumptions of the theorem we have that
$x_1\equiv \pm x_0p^s+p^{k-s}\inte_{p^s}~\left(\mbox{mod}~p^k\right)$  
and we see that the values of $x$ can be represented by the elements 
$$\left\{B(s)^z\left(\begin {array}{c} x_0p^s\\b\end {array}\right);
z=0,1,...,p^s-1\right\}\cup
\left\{B(s)^z\left(\begin {array}{c} -x_0p^s\\b\end {array}\right);
z=0,1,...,p^s-1\right\}$$ in $\inte_N^2.$
Here $B(s)=\left(\begin {array}{ccc} 1+rDp^{2(k-s)} & p^{k-s}D\\
\noalign{\medskip} p^{k-s} & 1+rDp^{2(k-s)}\end {array} \right)$ and by 
induction it is easy to show that 
$$B(s)^z=\left(\begin {array}{ccc} 1+rDz^2p^{2(k-s)} & 
\left(p^{k-s}z+3^{-1}rDp^{3(k-s)}(z^3-z)\right)D\\
\noalign{\medskip} p^{k-s}z+3^{-1}rDp^{3(k-s)}(z^3-z)
& 1+rDz^2p^{2(k-s)}\end {array} \right).$$ 
Denote $\zeta_{\pm,z}=\zeta_{B(s)^z
\left(\begin {subarray}{c} \pm x_0p^s\\b\end {subarray}\right)}$
and call the constants in front of these functions
$Ra_{\pm,z}.$ We have that 
$$\psi(b)=R\left( \sum_{z=0}^{p^s-1}a_{+,z}\zeta_{+,z}(b) +  \sum_{z=0}^{p^s-1} 
a_{-,z}\zeta_{-,z}(b)\right).$$ 
If we use Lemma~\ref{dynamik} we see that 
$U_N(B(s))\zeta_{\pm,z-1}=e\left(\frac{r\left(f_{\pm}(z)-f_{\pm}(z-1)\right)}
{N}\right)
\zeta_{\pm,z}$ for $z=1,...,p^s-1,$ where
\begin{align*}
f_{\pm}(z)&=\left(\pm\left(1+rDz^2p^{2(k-s)}\right)p^sx_0+\left(p^{k-s}z+
3^{-1}rDp^{3(k-s)}(z^3-z)\right)Db\right)\\
&\times\left(\pm\left(p^{k-s}z+3^{-1}rDp^{3(k-s)}(z^3-z)\right)p^sx_0 
+ \left(1+rDz^2p^{2(k-s)}\right)b\right)\\
&\equiv \pm p^sx_0b+p^{k-s}\left(Db^2+
p^{2s}x_0^2-p^{2(k-s)}3^{-1}rD^2b^2\right)z\\
&\pm p^{2k-s}2x_0Dbz^2+p^{3(k-s)}3^{-1}2D^2b^2z^3~~~~~~~~~~~~~~~~(\mbox{mod}~N).
\end{align*}
Since $B(s)^{p^s}=\left(\begin {array}{ccc} 1 & p^{k}D\\
\noalign{\medskip} p^{k} & 1 \end {array} \right)$ 
Corollary~\ref{egenfunktioner} gives us that 
$U_N(B(s))\psi=e\left(\frac{r\widetilde C}{p^{k+s}}\right)\psi$ for some 
$\widetilde C\equiv C~~\left(\mbox{mod}~p^k\right)$ and this leads to  
\begin{align*}
a_{\pm,z}
&=e\left(\frac{-r\widetilde C}{p^{k+s}}\right)
e\left(\frac{r\left(f_\pm(z)-f_\pm(z-1)\right)}{N}\right)a_{\pm,z-1}\\
&=e\left(\frac{-r\widetilde Cz}{p^{k+s}}\right)
e\left(\frac{r\left(f_\pm(z)-f_\pm(0)\right)}{N}\right)a_{\pm,0}.
\end{align*}
But
$\zeta_{\pm,z}(b)=
e\left(\frac{-p^{k+s}x_0^2z\mp p^{2k-s}3rx_0Dbz^2-p^{3(k-s)}rD^2b^2z^3}
{N}\right)$ hence
\begin{align*}
a_{\pm,z}\zeta_{+,z}(b)&=e\left(\frac{-p^{k-s}r\widetilde Cz+rf_{\pm}(z)
-rf_{\pm}(0)-p^{k+s}x_0^2z}{N}\right)\\
&\times e\left(\frac{\mp p^{2k-s}3rx_0Dbz^2-p^{3(k-s)}
rD^2b^2z^3}{N}\right)a_{\pm,0}
=a_{\pm,0}e\left(\frac{q_\pm(z)}{p^s}\right),
\end{align*}
where $q_\pm(z)=r\left(\Theta_\psi(b)z 
\pm x_0Dbz^2+p^{k-2s}3^{-1}D^2b^2z^3\right)$
and  
\begin{align}
\label{theta}
\Theta_\psi(b)p^k\equiv -x_0^2p^{2s}-\widetilde C + 
Db^2-p^{2(k-s)}3^{-1}rD^2b^2 ~~\left(\mbox{mod}~p^{k+s}\right).
\end{align}
\begin{remark*}
Note that $\Theta_\psi(b)$ is well defined, but that it can not be lifted to an 
integer polynomial. Different Hecke eigenfunctions in $V_C$ correspond to 
different choices of $\widetilde{C}\equiv C\left(\mbox{mod}~p^k\right).$  
\end{remark*}
\end{proof}

\begin{thm}
\label{Sats2}
Let $\psi\in V_C$ be a normalized Hecke eigenfunction for some 
$C\in \inte_{p^k}^\times.$ If $b\in \inte_N$ fulfills that 
$-C+Db^2\equiv 0~~\left(\mathrm{mod}~p^k\right)$ then  
\begin{align}
\label{uttryck2}
\psi(b)=\frac{\alpha_\psi(b)}{\sqrt{1-\left(\frac{D}{p}\right)\frac{1}{p}}}
\sum_{z=1}^{p^{[k/2]}}e\left(\frac{q(z)}{p^{[k/2]}}\right),
\end{align}
where $q(z)=r\left(\Theta_\psi(b)z+p^{k-2[k/2]}3^{-1}CDz^3\right),$ 
$|\alpha_\psi(b)|=1$ and 
$$\Theta_\psi(b)p^k\equiv -\widetilde C + Db^2-p^{k+(k-2[k/2])}3^{-1}rCD~~
\left(\mathrm{mod}~p^{[3k/2]}\right).$$
\end{thm}
\begin{proof}
This is the same proof as for Theorem~\ref{Sats1}.  
\end{proof}

\section{{\large Exponential sums of cubic polynomials}}
We have seen that the values of the Hecke eigenfunctions are given by 
exponential sums over rings $\inte_{p^s}.$ In this chapter we will derive the 
results we need in order to study the supremum of the eigenfunctions. 
For convenience we will still assume that $p>3.$

\begin{define}
Let $n$ be a nonnegative integer. For $q\in\inte_{p^n}[x]$
we define 
$$S(q,n)=\sum_{z=1}^{p^n}e\left(\frac{q(z)}{p^n}\right).$$
\end{define}

\begin{lemma}
\label{exakt exp}
Let $q(z)=a_3z^3+a_2z^2+a_1z+a_0$ and assume that $p|a_3$ but 
$p\notdivides a_2.$ Then 
$|S(q,n)|=p^{n/2}.$
\end{lemma}
\begin{proof}
It is trivial to see that $|S(q,0)|=1=p^{0/2}.$ On the other hand
$S(q,1)=\sum_{z=1}^{p}e\left(\frac{a_2z^2+ a_1z+a_0}{p}\right)$ and this 
Gauss sum is well known to have absolute value equal to $p^{1/2}$ 
(cf. \cite{Sch} chapter II.3). Now assume $n>1.$ The observation that we will 
use and use repeatedly is that if we have a polynomial $q\in \inte_{p^n}[z]$ then
$q(z_1+p^{n-1}z_2)\equiv q(z_1)+
q'(z_1)p^{n-1}z_2~~\left(\mbox{mod}~p^{n}\right).$ 
In this case this gives us that
\begin{align*}
S(q,n) &=  \sum_{z=1}^{p^n}e\left(\frac{q(z)}{p^n}\right)=
\sum_{z_1=1}^{p^{n-1}}\sum_{z_2=1}^{p} 
e\left(\frac{q(z_1+p^{n-1}z_2)}{p^n}\right)\\
&=\sum_{z_1=1}^{p^{n-1}}\sum_{z_2=1}^{p}
e\left(\frac{q(z_1)+q'(z_1)p^{n-1}z_2}{p^n}\right)\\
&=p\sum_{\substack{ z_1\in\inte_{p^{n-1}}\\
q'(z_1)\equiv 0~(\mbox{{\footnotesize mod}}~p)}}
e\left(\frac{q(z_1)}{p^n}\right)
=p\sum_{\substack{ z_1\in\inte_{p^{n-1}}\\z_1\equiv 
-a_1ra_2^{-1}~(\mbox{{\footnotesize mod}}~p)}}
e\left(\frac{q(z_1)}{p^n}\right)\\
&=p \sum_{z=1}^{p^{n-2}}e\left(\frac{q( -a_1ra_2^{-1}+zp)}{p^n}\right)
=p~e\left(\frac{q( -a_1ra_2^{-1})}{p^n}\right)S(q_1,n-2),
\end{align*}
where $q_1$ is a polynomial of degree 3 which fulfills the assumptions of the 
lemma. The proof now follows by induction.
\end{proof}

\begin{lemma}
\label{exp}
Let $q(z)=a_3z^3+a_1z+a_0$ and assume that $p\notdivides  a_3$ and that
$p^2\notdivides  a_1.$ Then $|S(q,n)|\le 2p^{n/2}.$ 
\end{lemma}
\begin{proof}
For $n=1$ this is well known, see for instance \cite{Sch}, therefore we 
assume that $n>1.$
Using the same calculation as in the proof of Lemma~\ref{exakt exp} we obtain 
that
\begin{align}
\label{summa}
S(q,n) = p\sum_{\substack{ z_1\in\inte_{p^{n-1}}\\
q'(z_1)\equiv 0~(\mbox{{\footnotesize mod}}~p)}}
e\left(\frac{q(z_1)}{p^n}\right).
\end{align}
The equation $q'(z_1)\equiv 0~(\mbox{mod}~p)$ 
has at most two solutions modulo $p,$ 
hence this expression consists of at most two different sums of length 
$p^{n-2}.$ If $p\notdivides a_1$ these sums are of 
the form $e\left(x_0p^{-n}\right)S(q_1,n-2),$ where $x_0\in \inte$ and 
$q_1$ fulfills the assumptions of Lemma~\ref{exakt exp}. On the other hand, if 
$a_1=\widetilde{a_1}p$ with $p\notdivides \widetilde{a_1},$ we get 
\begin{align*}
S(q,n) &=
p\sum_{\substack{ z_1\in\inte_{p^{n-1}}\\
q'(z_1)\equiv 0~(\mbox{{\footnotesize mod}}~p)}}
e\left(\frac{q(z_1)}{p^n}\right)
=p~e\left(\frac{a_0}{p^n}\right)\sum_{z_1=1}^{p^{n-2}}
e\left(\frac{a_3pz_1^3+\widetilde{a_1}z_1)}{p^{n-2}}\right)\\
&=p^2~e\left(\frac{a_0}{p^n}\right)
\sum_{\substack{z_1\in\inte_{p^{n-3}}\\
\widetilde{a_1}\equiv 0~(\mbox{{\footnotesize mod}}~p)}}
e\left(\frac{a_3pz_1^3+\widetilde{a_1}z_1)}{p^{n-2}}\right)=0.
\end{align*}
\end{proof}

\begin{lemma}
\label{exp2}
Let $q(z)=a_3z^3+p^2a_1z+a_0$ and assume that $p\notdivides a_3.$ 
For $n\ge 3$ we have that $|S(q,n)|=p^2 |S(q_1,n-3)|,$ 
where $q_1(z)=a_3z^3+a_1z.$
\end{lemma}
\begin{proof}
Again we write
\begin{align*}
S(q,n)&=
p\sum_{\substack{ z_1\in\inte_{p^{n-1}}\\
q'(z_1)\equiv 0~(\mbox{{\footnotesize mod}}~p)}}
e\left(\frac{q(z_1)}{p^n}\right)
=p~e\left(\frac{a_0}{p^n}\right)
\sum_{z_1=1}^{p^{n-2}}e\left(\frac{q_1(z_1)}{p^{n-3}}\right)\\
&=p^2~e\left(\frac{a_0}{p^n}\right)S(q_1,n-3).
\end{align*}
\end{proof}

\begin{define}
For $\alpha\in \inte_{p^n}^\times$ and $n=1$ or $n=2$
we define
$$A_{\alpha, n}=\frac{\sup_{t\in\inte_{p^n}}|S(q_{\alpha,t},n)|}{p^{n/2}},$$ 
where $q_{\alpha,t}(z)=\alpha z^3 +tz.$
\end{define}
\begin{remark*}
$A_{\alpha, n}$ is of course a function of $p$ but this is suppressed since we 
often think of $p$ as fixed.
\end{remark*}

\begin{lemma}
\label{A}
For fixed $n$ and $p$, $A_{\alpha, n}$ assumes at most three different values 
and if 
$p\equiv 2 ~(\mathrm{mod}~~3)$ then $A_{\alpha, n}$ is independent of $\alpha.$ 
Moreover, $1\le A_{\alpha,1}\le 2$ and $\sqrt{2}< A_{\alpha,2}\le 2.$
\end{lemma}
\begin{proof}
Since the multiplicative group $\inte_{p^n}^\times$ is cyclic of order 
$(p-1)p^{n-1}$ we write the elements as $g^k,$ where $k\in \inte_{(p-1)p^{n-1}}.$
If $p\equiv 2 ~(\mbox{mod}~3)$ then 3 is invertible in 
$\inte_{(p-1)p^{n-1}}$ so we 
see that $g^k=(g^{k/3})^3$ is a cube. If $p\equiv 1 ~(\mbox{mod}~3)$ any 
element can be written as $g^l\left(g^{k}\right)^3$ where $l=0,1,2.$ 
We have that 
$$A_{\alpha\beta^3, n}=\frac{\sup_{t\in\inte_{p^n}}|S(q_{\alpha\beta^3,t},n)|}
{p^{n/2}}=\frac{\sup_{t\in\inte_{p^n}}|S(q_{\alpha,t\beta^{-1}}(\beta z),n)|}
{p^{n/2}}=A_{\alpha, n}$$ 
and from this the first claim follows. To prove that 
$A_{\alpha,1}\ge 1$ we notice that $\left\{e\left(\frac{-tz}{p}\right)
\right\}_{t\in \inte_{p}}$ is an orthonormal basis in 
$L^2\left( \inte_{p}\right).$ Thus 
\begin{align*}
1&=\left\| e\left(\frac{\alpha z^3}{p}\right) \right\|_2^2
=\sum_{t\in \inte_{p}}\left| \left< e\left(\frac{\alpha z^3}{p}\right), 
 e\left(\frac{-tz}{p}\right)\right>\right|^2\\ 
&\le p \sup_{t\in\inte_{p}}\left| \frac{1}{p} S(q_{\alpha,t},1)\right|^2
=A_{\alpha, 1}^2.
\end{align*}
To prove that $A_{\alpha, 2}> \sqrt{2}$ we use the same proof but we notice that
we only have to sum over $t$ such that $S(q_{\alpha,t},2)\varnotsign=0.$ 
By the proof 
of Lemma~\ref{exp} we see that this gives us that $t\equiv 0~~(\mbox{mod}~p)$  
or that $t$ is a unit 
such that $-3^{-1}\alpha^{-1}t$ is a square (otherwise the sum in (\ref{summa}) 
is empty). The number of such $t$ is less than 
$p^2/2$ and that gives our estimate.
That $A_{\alpha, n}\le 2$ follows directly from Lemma~\ref{exp} and the fact 
that $|S(\alpha z^3,2)|=p.$
\end{proof}

\begin{thm}
\label{expsats}
If $q_{\alpha,t}(z)=\alpha z^3 +tz$ and $\alpha\in \inte_{p^n}^\times$
then 
$$\sup_{t\in\inte_{p^n}}|S(q_{\alpha,t},n)|=\left\{\begin{array}{ll}
p^{2n/3} & ~~~~\textrm{if~ $n\equiv 0~~(\mathrm{mod}~3)$}\\
A_{\alpha,1} p^{2n/3-1/6} & ~~~~\textrm{if~ $n\equiv 1~~(\mathrm{mod}~3)$}\\
A_{\alpha,2}p^{2n/3-1/3} & ~~~~\textrm{if~ $n\equiv 2~~(\mathrm{mod}~3)$}
\end{array}\right. .$$
\end{thm}
\begin{proof}
For $n=0,1,2$ the proof is trivial, hence assume $n\ge 3.$ We see that
$\sup_{t\in\inte_{p^n}}|S(q_{\alpha,t},n)|=\max\left( \sup_{p^2|t}|
S(q_{\alpha,t},n)|,\sup_{p^2\notdivides t}|S(q_{\alpha,t},n)| \right)$ 
and that the 
last of the two expressions is less than $2p^{n/2}$ by Lemma~\ref{exp}. The 
first expression is equal to $p^2\sup_{t\in\inte_{p^{n-3}}}|S(q_{\alpha,t},n-3)|$
by Lemma~\ref{exp2} and this is always larger than $2p^{n/2}$ 
since $\sqrt{p}>2.$ The theorem now follows by induction.
\end{proof}

\section{{\large Supremum norms of Hecke eigenfunctions in $V_C$}}
From \cite{KR} and \cite{Kurl2} we know that normalized Hecke 
eigenfunctions fulfill 
$$\| \psi \|_\infty
\le \frac{2}{\sqrt{1-\left(\frac{D}{p}\right)\frac{1}{p}}}$$
if $N=p$ and as we will see this is also true for $N=p^2$ 
(if $\psi$ is orthogonal to $S_2(1,1)$) and for ``half'' 
of the Hecke eigenfunctions 
for a general $N=p^{2k}.$ In fact, this estimate is a very good approximation of 
the supremum norm of these functions:
\begin{thm}
\label{sup1}
Let $N=p^{2k}$ for some prime $p>3$ that does not divide $C$ or $D$ and assume 
that $\psi\in V_C$ is a normalized Hecke eigenfunction.  
If $\left(\frac{C}{p}\right)=-\left(\frac{D}{p}\right)$ or if $k=1$ then 
$$\frac{2}{\sqrt{1-\left(\frac{D}{p}\right)\frac{1}{p}}}
\left(1-\frac{\pi^2}{8N}\right) \le \| \psi \|_\infty
\le \frac{2}{\sqrt{1-\left(\frac{D}{p}\right)\frac{1}{p}}}.$$
\end{thm}
\begin{proof}
We see that if $\left(\frac{C}{p}\right)=-\left(\frac{D}{p}\right)$ then 
$-C+Db^2\not \equiv 0~~(\mbox{mod}~p)$ for all $b,$ hence Theorem~\ref{Sats1} 
immediately gives 
\begin{align}
\label{8}
\| \psi \|_\infty\le \frac{2}{\sqrt{1-\left(\frac{D}{p}\right)\frac{1}{p}}}
\end{align}
in this situation. If $k=1$ then $s=0$ in Theorem~\ref{Sats1} and 
$[k/2]=0$ in Theorem~\ref{Sats2}, and this also gives the estimation (\ref{8}).
To prove the other inequality we pick $b\in \inte_N$ such that 
$\left(\frac{-C+Db^2}{p}\right)=1.$ 
We know (using the notation from the proof of 
Theorem~\ref{Sats1}) that 
\begin{align*}
\psi\left(b+tp^k\right)&= \frac{1}{\sqrt{1-\left(\frac{D}{p}\right)\frac{1}{p}}}
\left(a_{+,0}\zeta_{+,0}\left(b+tp^k\right)+a_{-,0}
\zeta_{-,0}\left(b+tp^k\right)\right)\\
&=\frac{1}{\sqrt{1-\left(\frac{D}{p}\right)\frac{1}{p}}}
\left(e\left(\frac{x_0t}{p^k}\right) a_{+,0}\zeta_{+,0}(b)
+e\left(\frac{-x_0t}{p^k}\right) a_{-,0}\zeta_{-,0}(b)\right)\\
&=\frac{e\left(\frac{-x_0t}{p^k}\right) a_{+,0}\zeta_{+,0}(b)}
{\sqrt{1-\left(\frac{D}{p}\right)\frac{1}{p}}}
\left( e\left(\frac{2x_0t}{p^k}\right)+\frac{a_{-,0}\zeta_{-,0}(b)}
{a_{+,0}\zeta_{+,0}(b)} \right).
\end{align*}
Since $x_0\not \equiv 0 ~~(\mbox{mod}~p)$
we can pick $t$ so that the difference 
$\theta$ of the arguments of the 
two expressions in the parenthesis is at most $\pi/p^k.$ Remembering that 
both the $a_{\pm,0}$ and $\zeta_{+,0}(b)$ have absolute value 1 we see that this 
$t$ gives us  
$$\left |\psi\left(b+tp^k\right)\right |=\frac{\sqrt{2+2\cos \theta}}
{\sqrt{1-\left(\frac{D}{p}\right)\frac{1}{p}}}
\ge\frac{ 2 - \frac{\theta^2}{4}}{\sqrt{1-\left(\frac{D}{p}\right)\frac{1}{p}}}
\ge \frac{2}{\sqrt{1-\left(\frac{D}{p}\right)\frac{1}{p}}}
\left(1-\frac{\pi^2}{8N}\right).$$
\end{proof}
The other ``half'' (neglecting $\bigoplus_{p|C}V_C$ for a moment) of the Hecke 
eigenfunctions have rather 
large supremum norms. As we shall see shortly these supremum norms assume at 
most three different values for a fixed $N.$
\begin{thm}
\label{supformel}
Let $N=p^{2k}$ for some prime $p>3$ and assume that $\psi\in V_C$ is a 
normalized Hecke eigenfunction for some 
$C\in \inte_{p^k}^\times.$ If 
$\left(\frac{C}{p}\right)=\left(\frac{D}{p}\right)$ and $k>1$ then
\begin{align}
\label{sup}
\| \psi \|_\infty=\frac{1}{\sqrt{1-\left(\frac{D}{p}\right)\frac{1}{p}}}\times
\left\{\begin{array}{ll}
p^{k/3}
 & ~~~~\textrm{if~ $k\equiv 0~~(\mathrm{mod}~3)$}\\
A_{36CD,2}~ p^{k/3-1/3}
& ~~~~\textrm{if~ $k\equiv 1~~(\mathrm{mod}~3)$}\\
A_{36CD,1}~ p^{k/3-1/6}
& ~~~~\textrm{if~ $k\equiv 2~~(\mathrm{mod}~3)$}\end{array}
\right. .
\end{align}
\end{thm}
\begin{proof}
Let us first estimate the expression in Theorem~\ref{Sats1},
that is equation (\ref{uttryck1}): If 
$b\equiv 0~~(\mbox{mod}~p)$ then 
$x^2\equiv -C+Db^2~~\left(\mbox{mod}~p^k\right)$ has at most 
2 different solutions and therefore we may assume that $b$ is a unit because 
otherwise $|\psi(b)|$ is much smaller than the expressions in equation 
(\ref{sup}). 
But then $|S(q_\pm,s)|=p^{s/2}$ by Lemma~\ref{exakt exp}, hence
$$|\psi(b)|\le \frac{2p^{s/2}}{\sqrt{1-\left(\frac{D}{p}\right)\frac{1}{p}}} 
\le \frac{2p^{(k-1)/4}}{\sqrt{1-\left(\frac{D}{p}\right)\frac{1}{p}}}.$$ 
We see that this is less than the claimed supremum norm if $k>2.$ If  
however $k=2$ then $s=0$ and using this we see that $|\psi(b)|$ is small also 
in this case. 
The expression in Theorem~\ref{Sats2} (equation (\ref{uttryck2})) 
has absolute value
$\left|\psi\left(b+tp^k\right)\right|
= \left(1-\left(\frac{D}{p}\right)\frac{1}{p}\right)^{-1/2}|S(q,[k/2])|$
where 
$$q(z)=r\left(\Theta_\psi\left(b+tp^k\right)z+
p^{k-2[k/2]}3^{-1}CDz^3\right).$$ 
By the definition of $\Theta_\psi$ we have that 
\begin{align*}
\Theta_\psi\left(b+tp^k\right)p^k & \equiv -\widetilde C+
D\left(b+tp^k\right)^2 - p^{k+(k-2[k/2])}3^{-1}rCD\\
&\equiv \left(\Theta_\psi (b) + 2Dbt\right)p^k~~
\left(\mbox{mod}~p^{[3k/2]}\right).
\end{align*}
Since $p\notdivides 2Db$ we see that, 
as we let $t$ run through all elements in $\inte_{p^k},$ the 
polynomial $q$ run through all polynomials of the form 
$q_\alpha(z)=\alpha z+p^{k-2[k/2]}3^{-1}rCDz^3$ with $\alpha\in\inte_{p^[k/2]}.$
We now study the cases when $k$ is even 
and when $k$ is odd separately:
If $k$ is odd we get $S(q_\alpha,[k/2])=0$ if $p\notdivides \alpha,$ hence 
$$\sup_{\alpha\in \inte_{p^{[k/2]}}}|S(q_\alpha,[k/2])|
= p\sup_{\alpha\in \inte_{p^{[k/2]-1}}}|S(w_\alpha,[k/2]-1)|,$$
where $w_\alpha(z)=\alpha z+3^{-1}rCDz^3.$ Applying Theorem~\ref{expsats} we 
get the expression we want. (Lemma~\ref{A} says that 
$A_{3^{-1}rCD,n}=A_{36CD,n}.$)
If $k$ is even we have that $q_\alpha=w_\alpha$ and we can apply 
Theorem~\ref{expsats} directly to get the desired expression.
\end{proof}

For completeness we also study the case when $p|D,$ that is the ramified case. 
Our evaluation procedure for the Hecke operators still works and we get the 
following result which is somewhat analogous to the known result for 
primes, see \cite{Kurl2}.

\begin{thm}
Let $\psi\in V_C$ be a normalized Hecke eigenfunction for some 
$C\in \inte_{p^k}^\times$ and assume that  $p|D.$ We have that
$$\sqrt{2}\left(1-\frac{\pi^2}{8N}\right) \le \| \psi \|_\infty
\le \sqrt{2}.$$
\end{thm} 

\begin{proof}
Let us determine the dimension of $V_C,$ that is the number of solutions 
to $x_1^2-Dx_2^2=-C$ in $\inte_{p^k}.$ This is easy because for any $x_2$ the 
equation $x_1^2=-C+Dx_2^2$ has exactly two solutions so the total number of 
solutions is $2p^k.$ We fix some $x_0$ such that $\mathscr{N}(x_0)=-C$ and we 
notice that every $x$ with $\mathscr{N}(x)=-C$ can be written as $hx_0$ for some 
$h\in H_D.$ This shows that $\psi$ is a sum of $\zeta_x-$functions where
$\mathscr{N}(x)=-C$ and the constants in front of them have absolute value 
$\sqrt{p^k/(2p^k)}=1/\sqrt{2}.$ 
We now argue as in the proof of Theorem~\ref{sup1} to 
get the desired conclusion. 
\end{proof}

Last we will turn our focus to the case when $p|C.$ 
This implies that $p$ is either 
split or ramified. 
The case when $p|C$ and $p$ is ramified will not be treated in this paper but 
one can expect that the supremum norms in that case behave in the same manner 
as in Theorem~\ref{supformel}. Now assume that $p$ is split
and let $\sqrt{D}$ be an 
element in $\inte_N$ such that $\sqrt{D}^2=D.$ Now define   
$$V_+=\bigoplus_{\begin{subarray}{c} x\in \inte_{p^k}^2\\ 
x_1\equiv\sqrt{D}x_2\not\equiv 0~(\mbox{{\footnotesize mod}}~p)
\end{subarray}}\C\zeta_x.$$
and $V_-$ in the same manner but with a minus sign in front of $\sqrt{D}.$ Note 
that $\bigoplus_{p|C}V_C=V_+\oplus V_-\oplus S_{2k}(2k-1,1)$ and that $V_\pm$ 
are invariant under the action of $H_D.$
\begin{thm}
\label{last}
Let $N=p^{2k}$ for some prime $p>3$ and assume that $p|C$ and that $D$ is 
a quadratic residue modulo $p.$ 
If $\psi\in V_C \cap V_\pm$ is a normalized Hecke eigenfunction then 
$$|\psi(b)|=\left\{\begin{array}{ll}
\frac{1}{\sqrt{1-\frac{1}{p}}} & ~~~~\textrm{if~ $p\notdivides b$}\\
0 &~~~~\textrm{if~ $p|b$}\end{array}
\right. .$$
\end{thm}

\begin{proof}
We may assume that $\psi\in V_C \cap V_+.$ 
To prove the theorem the main difficulty is to prove the following claim: 
If $\zeta_x,\zeta_y\in  V_C \cap V_+$ 
there is an $h\in H_D$ such that $hx\equiv y~~\left(\mbox{mod}~p^k\right).$
Assume that $p^l|C$ but $p^{l+1}\notdivides C.$ We see that 
$x_1\equiv \sqrt{D}x_2~~\left(\mbox{mod}~p^l\right)$ and that the same 
equality holds for $y.$ But then $p^l|x_1y_2-x_2y_1$ and we see that we can 
choose $h_2$ so that $-Ch_2\equiv x_1y_2-x_2y_1~~\left(\mbox{mod}~p^k\right).$
This determines $h_2$ modulo $p^{k-l}.$
Now choose $h_1\equiv (y_1-Dx_2h_2)x_1^{-1}~~\left(\mbox{mod}~p^k\right)$ and 
put $h=\left( \begin {array}{ccc} h_1 & h_2D\\
\noalign{\medskip} h_2 & h_1\end {array} \right).$ It is 
straightforward to verify that $hx\equiv y~~\left(\mbox{mod}~p^k\right),$ but 
in general $h\not\in H_D.$ In fact calculations show that $h_1^2-Dh_2^2
\equiv (y_1^2-(x_1y_2+x_2y_1)Dh_2)x_1^{-2}~~\left(\mbox{mod}~p^k\right)$
and we notice that the expression in front of $h_2$ is invertible. Since $h_2$
only is determined modulo $p^{k-l}$ we can choose $h_2$ so that $h\in H_D$ as 
long as we can show that $\det (h)\equiv 1~~\left(\mbox{mod}~p^{k-l}\right).$ 
But this follows immediately from the fact that $-C\equiv\mathscr{N}(y)
\equiv \mathscr{N}(hx)\equiv -C\det (h)~~\left(\mbox{mod}~p^k\right).$ 

Let $\psi\in V_C \cap V_+.$ The dimension of $V_C \cap V_+$ is $p^{k-1}(p-1),$ 
hence $\psi$ is a linear combination of $\zeta_x$ where the coefficients 
have absolute value $\sqrt{p^k/(p^{k-1}(p-1))}=(1-1/p)^{-1/2}.$ 
We see that if $p\notdivides b$ then 
$x^2\equiv -C+Db^2~~\left(\mbox{mod}~p^k\right)$ has exactly one solution such 
that $x\equiv\sqrt{D}b~~\left(\mbox{mod}~p\right)$ and if $p|b$ the equation has 
no solutions such that $x\not\equiv 0~~\left(\mbox{mod}~p\right).$ 
\end{proof}

\begin{remark*}
If $p|C$ and $\psi\in V_C$ is a normalized Hecke eigenfunction orthogonal 
to $S_{2k}(2k-1,1),$ then Cauchy-Schwarz inequality applied to Theorem~\ref{last}
gives us 
$$\| \psi \|_\infty \le \sqrt{\frac{2}{1-\frac{1}{p}}}.$$ 
\end{remark*}

\begin{thm}
\label{supupp}
Let $N=p^{2k}$ for some prime $p>3$ and assume that $p\notdivides D.$ 
If $\psi\in L^2\left(\inte_{N}\right)$
is a normalized Hecke eigenfunction then $\| \psi \|_\infty\le N^{1/4}.$  
\end{thm}
\begin{proof}
First assume that $p$ is inert. Then there is an integer $0\le m \le k$ 
such that $\psi\in S_{2k}(2k-m,m)$ but $\psi\not \in S_{2k}(2k-m-1,m+1).$ By  
Theorem~\ref{iso} $\psi\in S_{2k}(2k-m,m)\cong L^2
\left(\inte_{p^{2k-2m}}\right)$ 
and it is obvious that $T_m \psi$ must belong to $V_C$ for some 
$C\in \inte_{p^{2k-2m}}^\times .$  
Hence the estimates in 
Theorem~\ref{sup1} and Theorem~\ref{supformel} together with 
the fact that $T_m$ is unitary gives the estimate 
directly.
Now assume that $p$ is split. If $\psi\in V_C$ for some 
$C\in \inte_{p^{2k}}^\times$ then Theorem~\ref{sup1} and Theorem~\ref{supformel}
gives the estimate. If $\psi\in V_C$ and $p|C$ we write 
$\psi=\psi_0+\psi_1+...+\psi_k,$ where $\psi_m\in S_{2k}(2k-m,m)$ but 
$\psi_m$ is orthogonal to $S_{2k}(2k-m-1,m+1).$ Theorem~\ref{last} together with 
Theorem~\ref{iso} tells us that the support of $\psi_m$ is 
$\{x;p^m|x \wedge p^{m+1}\notdivides x\},$ 
hence the supports are all disjoint and we see that 
$\| \psi \|_\infty= \max_{0\le m \le k}\| \psi_m \|_\infty.$ 
By our last remark we see that 
$$\| \psi_m \|_\infty\le \sqrt{\frac{2}{1-\frac{1}{p}}}~ p^{m/2}\| \psi_m \|_2
\le \sqrt{\frac{2}{1-\frac{1}{p}}}~p^{m/2}$$ 
for $m<k$ and $\| \psi_k \|_\infty=p^{k/2}\| \psi_k \|_2\le p^{k/2}.$
\end{proof}

\begin{remark*}
Note that Theorem~\ref{supupp} is true for all $N'$ that could be written as a 
product of different $N$ of the form stipulated in the theorem.  
Also note that the estimates $|\psi(x)| \le \| \psi \|_\infty\le N^{1/4}$ 
implies that $h(\psi )\ge \frac{1}{2}\log N,$ the estimate in 
Theorem~\ref{entropi}.
\end{remark*}


\end{document}